\documentclass[times]{article}

\usepackage{framed,multirow}

\usepackage{graphicx}
\usepackage{amssymb}
\usepackage{latexsym}
\usepackage{amsmath}
\usepackage{color}
\usepackage{undertilde}
\usepackage[toc,page]{appendix}
\usepackage{bbm}
\usepackage{amsthm}
\usepackage{url}
\usepackage{xcolor}
\definecolor{newcolor}{rgb}{.8,.349,.1}

\def\undertilde#1{\mathord{\vtop{\ialign{##\crcr
$\hfil\displaystyle{#1}\hfil$\crcr\noalign{\kern1.5pt\nointerlineskip}
$\hfil\widetilde{}\hfil$\crcr\noalign{\kern1.5pt}}}}}
\newcommand{\xx}{\mathbf{x}}
\newcommand{\ww}{\mathbf{w}}

\newcommand{\uu}{\mathbf{u}}
\newcommand{\vv}{\mathbf{v}}
\newcommand{\zz}{\mathbf{z}}
\newcommand{\va}{\mathbf{a}}
\newcommand{\vb}{\mathbf{b}}
\newcommand{\vc}{\mathbf{c}}

\newcommand{\hl}{\hat{\ell}}

\newcommand{\bk}{\mathbf{\kappa}}

\newcommand{\XX}{\mathbf{X}}
\newcommand{\ZZ}{\mathbf{Z}}

\newcommand{\CC}{\mathbf{C}}
\newcommand{\UU}{\mathbf{U}}
\newcommand{\VV}{\mathbf{V}}

\newcommand{\LLam}{\mathbf{\Lambda}}
\newcommand{\SSig}{\mathbf{\Sigma}}
\newcommand{\mX}{\utilde{\XX}}
\newcommand{\mx}{\utilde{\xx}}
\newcommand{\mC}{\utilde{\CC}}

\newcommand{\mV}{\utilde{\VV}}

\newcommand{\mLam}{\utilde{\LLam}}
\newcommand{\mlam}{{\utilde{{\lambda}}}}
\newcommand{\ns}{n_{\textsf{s}}}

\newcommand{\new}{\textsf{new}}
\newcommand{\TT}{\!\textsf{T}}
\newcommand{\GG}{\mathbf{G}}
\newcommand{\mG}{\utilde{\GG}}
\newcommand{\BB}{\mathbf{B}}
\newcommand{\tLLam}{\widetilde{\mathbf{\Lambda}}}
\newcommand{\mtLLam}{\widetilde{\mLam}}
\newcommand{\bg}{\mathbf{g}}
\newcommand{\mg}{\utilde{\bg}}
\newcommand{\ones}{\mathbbmss{1}}
\newtheorem{prop}{Proposition}
\newcommand\RR{\leavevmode\hbox{$\rm I\!R$}}
\newcommand{\NewText}[1]{{\color{black} #1}}

\begin{document}
\title{A kernel Principal Component Analysis (kPCA) \emph{digest} with a new backward mapping (pre-image reconstruction) strategy}

\author{Alberto  Garc\'ia-Gonz\'alez, Antonio Huerta, Sergio Zlotnik and Pedro D\'iez\\
Laboratori de C\`alcul Num\`eric, E.T.S. de Ingenier\'ia de Caminos, \\
Universitat Polit\`ecnica de Catalunya -- BarcelonaTech}

\maketitle 

\begin{abstract}
Methodologies for multidimensionality reduction aim at discovering low-dimensional manifolds where data ranges. Principal Component Analysis (PCA) is very effective if data have linear structure. But fails in identifying a possible dimensionality reduction if data belong to a nonlinear low-dimensional manifold. For nonlinear dimensionality reduction, kernel Principal Component Analysis (kPCA) is appreciated because of its simplicity and ease implementation. The paper provides a concise review of PCA and kPCA main ideas, trying to collect in a single document aspects that are often dispersed. Moreover, a strategy to map back the reduced dimension into the original high dimensional space is also devised, based on the minimization of a discrepancy functional.
\end{abstract}
%
%\begin{keyword}
%\MSC 6207\sep %Data analysis
%68T10\sep %Pattern recognition
%62H25 %\sep %Nonlinear multidimensionality reduction 
%
%\KWD Nonlinear multidimensionality reduction \sep 
%Principal Component Analysis \sep
%kernel Principal Component Analysis \sep
%data pre-image \sep
%backward mapping
%
%%% MSC codes here, in the form: \MSC code \sep code
%%% or \MSC[2008] code \sep code (2000 is the default)
%\end{keyword}

%\linenumbers

%% main text

\section{Introduction}\label{sec:Intro}
The kernel Principal Component Analysis (kPCA) is a very effective and popular technique to perform nonlinear dimensionality reduction \cite{Scholkopf1998}. It is applied to a large variety of fields: image processing and signal denoising \cite{rathi2006statistical,twining2001kernel,Mika1999}, face recognition \cite{Wang2012}, credible real-time simulations in engineering applications \cite{gonzalez2018kpca,lopez2018manifold,PD-DZGH-18}, health and living matter sciences \cite{Yuka2018,Devy2012,Grassi2014}... kPCA is often presented giving the general ideas and avoiding the fundamental mathematical details. These details are soundly described only in few selected papers. In particular, when it comes to the forward mapping (or dimensionality reduction, see  \cite{Alan2008}) and to the backward mapping or pre-image reconstruction, see \cite{Snyder2013}.

This paper aims at providing a kPCA \emph{digest}. That is, a condensed description with all the details necessary to understand and implement the method, alternative and complementary to the existing literature. 
Besides, motivated by different alternatives to achieve the backward mapping (or pre-image reconstruction) available in the literature \cite{kwok2004,Mika1999,Wang2012,Scholkopf1998kernel,rathi2006statistical,zheng2010penalized}, a new approach is proposed based on the minimization of a discrepancy functional (residual). This novel idea allows dealing with the very frequently used case of Gaussian kernel in a straightforward and efficient manner.

\section{Principal Component Analysis (PCA)}\label{sec:PCA}
\subsection{Diagonalizing the covariance matrix}\label{sec:PCA1}
Vectors $\xx^{1}, \xx^{2}, \dots, \xx^{\ns}$ in $\RR^{d}$ are seen as $\ns$ samples of some (possibly random) variable $\xx \in \RR^{d}$. It is assumed that $\ns \gg d$ such that the sample is representative of the variability of  $\xx$. The samples are collected in a $d\times \ns$ matrix $\XX=[\xx^{1} \, \xx^{2} \, \cdots  \, \xx^{\ns}]$. 

PCA  aims at discovering a linear model of dimension $k$, being $k \ll d$, properly representing the variability of $\xx$. Thus, eliminating the intrinsic redundancy in the $d$ dimensions of vector $\xx$. 

The covariance $d \times d$ matrix is defined as 
\begin{equation}\label{eq:MatrixC}
\CC =\XX \XX^{\TT}=\sum_{\ell=1}^{\ns} \xx^{\ell} (\xx^{\ell})^{\TT} .
\end{equation}
In fact, the symmetric and positive definite matrix  $\CC$ stands for the covariance matrix of $\xx$ if the mean value of $\xx$ is zero.  This is not a a loss of generality because it  simply  requires subtracting the mean value of $\xx$ to all the columns in $\XX$.
\NewText{In other words, this is a pre-process that consists in centering the samples.  This straightforward operation maps the set of samples to a neighborhood of the origin (around the zero). PCA aims at fitting a linear manifold (thus, containing the element zero) to the data. Therefore, using the centered data is easing the task of PCA. If the data is not centered, the element zero is generally far from the training set. Thus, in order to account for the affine character of the data set, PCA with non-centered data typically requires one additional dimension in the reduced space to obtain the same level of accuracy.}

Diagonalizing $\CC$ results in finding a diagonal matrix $\LLam$ of eigenvalues $\lambda_{1}\ge\lambda_{2}\ge \cdots \ge \lambda_{d}\ge0$  and a unit matrix $\UU$ such that 
\begin{equation}\label{eq:DiagC}
\CC = \UU \LLam \UU^{\TT} .
\end{equation}
Matrix $\UU$ describes an orthonormal basis in $\RR^{d}$, the basis formed by its columns $\UU=[\uu^{1} \, \uu^{2} \, \cdots  \, \uu^{d}]$. 

Vector $\zz = \UU^{\TT} \xx$ is the expression of $\xx$ in this new basis (note that $\xx = \UU \zz$). Thus, matrix $\ZZ = \UU^{\TT} \XX$ is the corresponding matrix of samples of $\zz$. The covariance of $\zz$ is 
$\ZZ \ZZ^{\TT} =   \LLam$.
%provided by matrix 
%\begin{equation}\label{}
%\ZZ \ZZ^{\TT} =  \UU^{\TT} \XX  \XX^{\TT} \UU = \UU^{\TT}  \UU \LLam \UU^{\TT}  \UU = \LLam
%\end{equation}
Being the covariance of $\zz$ diagonal, the $d$ components of $\zz$ are uncorrelated (linearly independent) random variables.

%%%%%%%%%%%%%%%%%%%%%%%%%%%%%%%%%%%%\vspace{0.5cm}
\subsection{Reducing the dimension}\label{sec:PCA2}
%\vspace{0.5cm}

The trace of $\CC$, which is equal to the trace of $\LLam$ and therefore the sum of $\lambda_{i}$,  for $i=1,\ldots,d$, is the total variance of the vector sample. If eigenvalues decrease quickly, a reduced number of dimensions $k$ is collecting a significant amount of the variance. This is the case if, for some small tolerance $\varepsilon$, $k$ is such that 
\begin{equation}\label{eq:reducedk}
\sum_{i=1}^{k} \lambda_{i}  \ge (1 - \varepsilon ) \sum_{i=1}^{d} \lambda_{i} 
\end{equation}
Thus, eigenvalues from $k+1$ to $d$ are neglected, and consequently the last $d-k$ columns of matrix $\UU$ 
(or the last rows of matrix $\UU^{\TT}$)
are not expected to contribute to describe the variability of $\xx$.
Accordingly, the last $d-k$ components of vector $\zz$ are suppressed without significant loss of information.

Thus, $d \times k$ matrix $\UU^{\star}=[\uu^{1} \, \uu^{2} \, \cdots  \, \uu^{k}]$, induces a new variable 
\begin{equation}\label{eq:zstar}
\zz^{\star} := \UU^{\star\TT} \xx \in\RR^{k}
\end{equation}
in the reduced-dimension space. 
%Each value of $\xx\in\RR^{d}$ is straightforwardly mapped into the reduced-dimensional space $\RR^{k}$ by doing $\zz^{\star}=\UU^{\star} \xx$.
Thus,  the samples in $\XX$ map into $\ZZ^{\star}=\UU^{\star\TT} \XX$ of reduced dimension $k\times \ns$.

The backward mapping (from reduced-dimension space $\RR^{k}$ to full dimension $\RR^{d}$) can be seen as a truncation of relation $\xx=\UU \zz$, that is rewritten as
\begin{equation}\label{eq:Backward1}
\xx = \sum_{i=1}^{d} [\zz]_{i} \uu^{i} \approx  \sum_{i=1}^{k} [\zz]_{i} \uu^{i}, 
\end{equation}
being $[\zz]_{i}$, the $i$-th component of vector $\zz$. In matrix form when restricted to the samples, this reads
\begin{equation}\label{eq:Backward2}
\XX = \UU \ZZ \approx  \UU^{\star} \ZZ^{\star}.
\end{equation}
The reduced-order models is interpreted as taking $\xx$ in the $k$-dimensional linear manifold (or subspace) generated by the basis $\{\uu^{1} , \uu^{2} , \cdots  , \uu^{k}\}$. The error introduced in the reduction of dimensionality (from $d$ to $k$) is associated with the discrepancy $\XX - \UU^{\star} \ZZ^{\star}$ and is decreasing as the tolerance $\varepsilon$ decreases (and $k$ increases).
%%%%%%%%%%%%%%%%%%%%%%%%%%%%%%%%%%%%%%%%%%%%%%%%
%\vspace{0.5cm}
\subsection{Singular Value Decomposition (SVD): an alternative to diagonalization.}\label{sec:SVD}
%\vspace{0.5cm}

The SVD provides a factorization of $d\times \ns$ matrix $\XX$ of the form
\begin{equation}\label{eq:SVD1}
\XX = \UU \SSig \VV^{\TT}
\end{equation}
being $\UU$ and $\VV$ unit matrices of sizes $d\times d$  and $\ns \times \ns$, respectively, and $\SSig$ a \emph{diagonal} $d\times \ns$ matrix. The singular values  of $\XX$, $\sigma_{1}\ge\sigma_{2}\ge \cdots \ge \sigma_{d}\ge0$ are the diagonal entries of $\SSig$ (recall $\ns \ll d$).
\begin{equation*}\label{e:sdiag}
\SSig = 
\left[
\begin{array}{ccccccc}
       \sigma_1 & & & & 0 & \hdots & 0 \\
       &  \sigma_2  & & &0 & \hdots & 0 \\
       & &  \ddots & & 0 &  \hdots & 0  \\
       & & &  \sigma_{d} & 0 & \hdots & 0 
       \end{array} 
\right]  
\end{equation*}

Note that the diagonalization of $\CC$ is a direct consequence of the SVD
\begin{equation*}\label{eq:eq1}
\CC= \XX \XX^{\TT} %= \UU \SSig \VV^{\TT} \VV \SSig^{\TT} \UU^{\TT} 
=  \UU [\SSig  \SSig^{\TT}] \UU^{\TT} =  \UU \LLam \UU^{\TT} 
\end{equation*}
being $\LLam = \SSig  \SSig^{\TT}$. Thus, the eigenvalues of $\CC$ are precisely the squared singular values of $\XX$, that is $\lambda_{i}=\sigma_{i}^{2}$,  for $i=1,\ldots,d$.

%\vspace{0.5cm}
\subsection{Permuting $d$ and $\ns$.}\label{sec:permuting}
%\vspace{0.5cm}

The $\ns\times \ns$ matrix equivalent to $\CC$ taking $\XX^{\TT}$ as input matrix (organizing data by rows instead of columns), is
\begin{equation}\label{eq:MatrixG}
\GG=\XX^{\TT} \XX \,\text{ such that } [\GG]_{ij}=(\xx^{i})^{\TT} \xx^{j}. 
\end{equation}
It is denoted as Gramm matrix because it accounts for all the scalar products of the samples. This perspective of the problem aims at a reduction of the number of samples $\ns$, while keeping the representativity of the family. This methodology is often named as Proper Orthogonal Decomposition (POD): it coincides with PCA, but aiming at reducing the number of samples $\ns$  instead of the dimension $d$.

However, the fact of changing $\XX$ by $\XX^{\TT}$ is not relevant when using SVD. Namely, the SVD \eqref{eq:SVD1} reads
\begin{equation*}\label{eq:eq2}
\XX^{\TT} = \VV \SSig^{\TT} \UU^{\TT} .
\end{equation*}
And directly provides a diagonalization of the \emph{large} ($\ns \times \ns$) Gramm matrix $\GG=\XX^{\TT} \XX$ 
\begin{equation}\label{eq:eq3}
\GG=  \VV [\SSig^{\TT}  \SSig] \VV^{\TT} = \VV \tLLam  \VV^{\TT} ,
\end{equation}
where the diagonal $\ns \times \ns$ matrix $\tLLam = \SSig^{\TT}  \SSig$ has the same nonzero entries 
$\lambda_{i}$, for $i=1,\ldots,d$, as $\LLam$.

The dimensionality reduction is performed exactly in the same way as described in section \ref{sec:PCA2}, allowing to reduce the dimension of the space of the samples from $\ns$ to $k$. That is, to describe with enough accuracy any of the $\ns$ samples (or any other vector pertaining to the same family) as a linear combination of $k$ vectors which are precisely the first $k$ columns of matrix $\VV$.

\subsection{The equivalence of diagonalizing $\CC$ and  $\GG$}\label{sec:PCA3}
When diagonalizing the covariance matrix $\CC$ as indicated in \eqref{eq:DiagC}, the columns of the transformation matrix $\UU$ are precisely the eigenvectors of $\CC$, that is
\begin{equation}\label{eq:eig1}
\CC \uu^{i}=\lambda_{i} \uu^{i} \,\,\text{ for } i=1,\dots,d \, .
\end{equation}
The same happens with the Gramm matrix $\GG$ and matrix $\VV$ in \eqref{eq:eq3}, namely
\begin{equation}\label{eq:eig2}
\GG \vv^{i}=\lambda_{i} \vv^{i} \,\,\text{ for } i=1,\dots,d \, ,
\end{equation}
being $\vv^{i}$ the $i$-th column of $\VV$. It is worth noting that for the for $i=d+1,\dots,\ns$, the columns $\vv^{i}$ of matrix $\VV$ correspond to the zero eigenvalue (they describe the kernel space of the matrix).

The fact that the SVD described in section \ref{sec:SVD} provides in one shot both $\UU$ and $\VV$ suggests that the computational effort provided in diagonalizing $\CC$ (and therefore obtaining $\UU$, see \eqref{eq:DiagC}) is equivalent to the cost of diagonalizing $\GG$ (and obtaining $\VV$). This is not obvious at the first sight, because the size of $\CC$ is $d\times d$ and the size of $\GG$ is $\ns\times \ns$, hence much larger.  

This equivalence actually holds. Computing vectors $\uu^{i}$,  for $i=1,\dots,d$, is equivalent to compute vectors $\vv^{i}$. This idea is one of the cornerstones of the kernel Principal Component Analysis (kPCA) method.

Recalling \eqref{eq:MatrixC}, isolating $\uu^{j}$ from the right-hand-side of  \eqref{eq:eig1}, and using Appendix \ref{app2}, yields
\begin{multline}\label{eq:ueqsBx}
\uu^{j}= \frac{1}{\lambda_{j}}  \sum_{\ell=1}^{\ns}  (\xx^{\ell} \xx^{\ell \, \TT}) \uu^{j} 
          =  \sum_{\ell=1}^{\ns} \frac{1}{\lambda_{j}}  (\xx^{\ell \, \TT}  \uu^{j}) \xx^{\ell}
          = \sum_{\ell=1}^{\ns} [\BB]_{\ell j} \xx^{\ell}
\end{multline}
where, assuming that $\UU$ is available, the $\ns \times d$ matrix $\BB$ is introduced such that the entry with indices $\ell$, $j$, for $ \ell=1,\dots,\ns$ and $j=1,\dots,d$, reads
\begin{equation}\label{eq:matB}
[\BB]_{\ell j} =\frac{1}{\lambda_{j}} \xx^{\ell \, \TT} \uu^{j}.
\end{equation}

Matrix $\BB$ is computable and coincides precisely with the first $d$ columns of matrix $\VV$, see Appendix \ref{app:proof1}, that is
\begin{equation}\label{eq:matBeqsmatV}
[\BB]_{\ell j} = [\vv^{j}]_{\ell}.
\end{equation}

The reduced variable $\zz^{\star}$ of dimension $k$ is introduced in \eqref{eq:zstar} . For the samples in the \emph{training set}, that is for $\xx^{j}$, $j=1,\dots,\ns$, the $i$-th component of the reduced-dimensional samples read
\begin{align}\label{eq:zstar2}
[\zz^{j \,\star}]_{i} = (\uu^{i})^{\TT} \xx^{j} &=  (\sum_{\ell=1}^{\ns} [\BB]_{\ell i} \xx^{\ell})^{\TT} \xx^{j} \nonumber \\ 
&=\sum_{\ell=1}^{\ns} [\BB]_{\ell i} ((\xx^{\ell})^{\TT} \xx^{j}) = [\VV^{\star \, \TT} \GG]_{i j} 
\end{align}
for $i=1,\dots,k\le d$.
That is, the compact expression for the construction of the $k \times \ns$ matrix $\ZZ^{\star}$ of samples in the reduced space reads 
\begin{equation}\label{eq:final}
\ZZ^{\star}=
\underset{(k \times \ns)}{\VV^{\star\TT}}
\underset{(\ns \times \ns)}{\GG} .
\end{equation}
The backward mapping described in \eqref{eq:Backward2} is also written in terms of $\VV^{\star}$, namely
\begin{equation}\label{eq:backMap}
\underset{(d \times \ns)}{\XX} \approx \underset{(d \times k)}{\UU^{\star}}  \underset{(k \times \ns)}{\ZZ^{\star}}
= \XX  \underset{(\ns \times k)}{\VV^{\star}}  \underset{(k \times \ns)}{\ZZ^{\star}}
\end{equation}
Equation \eqref{eq:backMap} stands because a straightforward rearrangement of \eqref{eq:ueqsBx} taking into account \eqref{eq:matBeqsmatV} shows that 
\begin{equation}\label{eq:UeqXV}
\UU^{\star} = {\XX}  \VV^{\star} .
\end{equation}
Note that, when compared with equation \eqref{eq:Backward2}, in \eqref{eq:backMap} the unknown $\XX$ appears also in the right-hand-side when it has to be expressed in terms of $ \VV^{\star}$ (when  $\UU^{\star}$ is not available, as it is the case in section \ref{sec:backwardkPCA}). This is one of the difficulties to overcome in order to devise a backward mapping strategy for kPCA.

%%%%%%%%%%%%%%%%%%%%%%%%%%%%%%%%%%%%%%%%%%%%%%%
\section{kernel Principal Component Analysis}\label{sec:kPCA}

As shown above, PCA is a dimensionality reduction technique that, for some $d$-dimensional variable $\xx$,  discovers linear manifolds of dimension $k\ll d$  where $\xx$ lies (within some tolerance $\varepsilon$). 
%At least, where the $\ns$ samples available range. 
In many applications, reality is not as simple and the low-dimensional manifold characterizing the variable is nonlinear. In these cases PCA fails in identifying a possible dimensionality reduction. Several techniques allow achieving nonlinear dimensionality reduction, see \cite{bishop:2006:PRML} and references therein. Among them, kernel Principal Component Analysis (kPCA) \cite{Scholkopf1998,rathi2006statistical} is appreciated for its simplicity and easy implementation. 

\subsection{Transformation to higher-dimension $D$}\label{sec:transfromPhi}
The kPCA conceptual idea is conceived by introducing an arbitrary transformation $\Phi$ from $\RR^{d}$ to $\RR^{D}$ for some \emph{very large dimension} $D\gg d$. This transformation into a large dimensional space is expected to \emph{untangle} (or to \emph{flatten}) the nonlinear manifold where $\xx$ belongs. That is, 
 $\Phi$ and $D$ are expected to be 
such that the variable $\mx = \Phi(\xx)$ mapped into $\RR^{D}$, lies in a linear low-dimensional manifold that is readily identified by the PCA. In other words, PCA is to be applied to the $D\times \ns$ matrix containing the transformed samples
\begin{equation}\label{eq:transfSamples}
\mX {=} [\Phi(\xx^{1}) \, \Phi(\xx^{2}) \, \cdots  \, \Phi(\xx^{\ns})] 
       {=} [\mx^{1} \, \mx^{2} \, \cdots  \, \mx^{\ns}] . 
\end{equation}

Assuming that $\Phi$ (and therefore $D$) are known, the standard version of PCA is applied; it diagonalizes the $D\times D$ covariance matrix $\mC = \mX \mX^{\TT}$, as described in section \ref{sec:PCA}. This introduces two obvious difficulties: first, $\Phi$ is unknown and, second, the dimension $D$ required to properly untangle the manifold is typically very large (in particular much larger than $\ns$) and consequently the computational effort to diagonalize $\mC$ is unaffordable.

Note however that the transformed Gramm matrix $\mG=\mX^{\TT} \mX$ is of size $\ns \times \ns$ (same as $\GG$).  Moreover, as shown in section \ref{sec:permuting}, diagonalizing $\mG$ reduces the dimension in the same way as diagonalizing $\mC$. 

Thus, in practice, the transformation $\Phi$ is mainly required to compute $\mG$. 

\subsection{The kernel trick}\label{sec:kernelTrick}

As noted above, the a priori unknown $\Phi$ is expected to map the nonlinear manifold into a linear one in a higher-dimensional space. In principle, it is not obvious to determine a proper $\Phi$ that \emph{aligns} the samples into a linear subspace of $\RR^{D}$. 
% but it is a priori unknown. It has to be somehow invented and one would expect to make several trials until discovering a proper $\Phi$ that \emph{aligns} the samples into a linear subspace of $\RR^{D}$. 

The kernel trick consists in defining directly $\mG$. Note that the generic term of matrix $\mG$ reads
\begin{equation}\label{eq:matmG1}
[ \mG]_{ij} = \Phi(\xx^{i})^{\TT} \Phi(\xx^{j})=(\mx^{i})^{\TT} \mx^{j} 
\end{equation}
for $i,j=1,\dots,\ns$. The kernel idea is to introduce, instead of function $\Phi(\cdot)$, some bivariate  symmetric form $\bk(\cdot,\cdot)$ such that 
\begin{equation}\label{eq:matmG2}
[ \mG]_{ij} =\bk(\xx^{i},\xx^{j}) .
\end{equation}
Comparing \eqref{eq:matmG1} and \eqref{eq:matmG2} reveals that selecting some $\bk(\cdot,\cdot)$ is, for all practical purposes, equivalent to selecting some $\Phi(\cdot)$.

A typical choice for the kernel is the Gaussian that reads
\begin{equation}\label{eq:Gausskernel}
\bk(\xx^{i},\xx^{j})= \exp(-\beta\Vert \xx^{i}-\xx^{j}\Vert^{2}) .
\end{equation}
Any other choice is valid provided that the matrix $\mG$ generated with the sample is symmetric and positive definite (recall that it must be a Gramm matrix). 
Many other alternative kernels are proposed in the literature \cite{Scholkopf1998,rathi2006statistical,twining2001kernel,Mika1999}. 
The idea is to select the kernel providing the best dimensionality reduction (the lower value of $k$).

\subsection{Centering the kernel}\label{sec:Centeringkernel}
As mentioned above, 
\NewText{centering the samples is a best practice  for improving the efficiency of PCA}. 
Note however that selecting some arbitrary $\Phi(\cdot)$ does not guarantee that the transformed sample remains centered. However, if $\Phi(\cdot)$ is available, the operation to center $\mX$ in \eqref{eq:transfSamples} is straightforward. Similarly,  selecting $\bk(\cdot,\cdot)$ produces a matrix $\mG$ which is, in general, not centered.

If $\Phi$ was known, centering the transformed sample requires redefining $\mx^{i}$ 
(changing $\mx^{i}$ into ${\overline \mx}^{i}$) as
\begin{equation}\label{eq:centerPhi}
{\overline \mx}^{i} = \Phi(\xx^{i})-\frac{1}{\ns} \sum_{\ell=1}^{\ns} \Phi(\xx^{\ell}).
\end{equation}
In the following, the centered magnitudes are denoted with an over bar, like ${\overline \mx}^{i}$ in \eqref{eq:centerPhi}. 

Thus, the generic term of the centered Gramm matrix, analogously to \eqref{eq:matmG1}, reads
$[\overline{\mG}]_{ij}=({\overline \mx}^{i})^{\TT} {\overline \mx}^{j}$, that results in 
\begin{multline*}%\label{eq:centeredmG1}
[\overline{\mG}]_{ij} = 
\Phi(\xx^{i})^{\TT}\Phi(\xx^{j})  {-}  \frac{1}{\ns} \sum_{\hl=1}^{\ns}  \Phi(\xx^{i})^{\TT}\Phi(\xx^{\hl}) \\
  {-} \frac{1}{\ns} \sum_{\ell=1}^{\ns} \Phi(\xx^{\ell})^{\TT} \Phi(\xx^{j})  
  {+}     \frac{1}{\ns^{2}} \sum_{\ell=1}^{\ns} \sum_{\hl=1}^{\ns} \Phi(\xx^{\ell})^{\TT}\Phi(\xx^{\hl})
\end{multline*}
or
\begin{equation*}
[\overline{\mG}]_{ij} =   \bk(\xx^{i},\xx^{j}){-}\frac{1}{\ns}\sum_{\hl=1}^{\ns}  \bk(\xx^{i},\xx^{\hl})% \\[-1ex]
   {-} \frac{1}{\ns} \sum_{\ell=1}^{\ns} \bk(\xx^{\ell},\xx^{j})
   {+}  \frac{1}{\ns^{2}} \sum_{\ell=1}^{\ns} \sum_{\hl=1}^{\ns} \bk(\xx^{\hl},\xx^{\ell})  
\end{equation*}
Recalling \eqref{eq:matmG2}, this matrix results as
\begin{equation}\label{eq:centerdmG2}
\overline{\mG}=\mG{-}\frac{1}{\ns} \mG  \ones_{_{[\ns\!{\times}\!\ns]}} 
{-} \frac{1}{\ns}  \ones_{_{[\ns\!{\times}\!\ns]}} \mG  
{+} \frac{1}{\ns^{2}}  \ones_{_{[\ns\!{\times}\!\ns]}} \mG  \ones_{_{[\ns\!{\times}\!\ns]}} 
%\overK=\KK -\KK[\frac{1}{\ns}]^{\ns \times \ns}   -[\frac{1}{\ns}]^{\ns \times \ns}  \KK     +   [\frac{1}{\ns}]^{\ns \times \ns}   \KK [\frac{1}{\ns}]^{\ns \times \ns}
\end{equation}
being $\ones_{_{[\ns\!{\times}\!\ns]}} \in \RR^{\ns\times \ns}$ the $\ns{\times}\ns$ matrix having all its entries equal to one. Note that expression \eqref{eq:centerdmG2} provides the centered Gramm matrix $\overline{\mG}$ after a simple algebraic manipulation of matrix $\mG$ directly generated by kernel $\bk(\cdot,\cdot)$.

Accordingly, column ${\overline \mg}^{j}$ of matrix $\overline{\mG}$ is given by
\begin{equation}\label{eq:centcol}
{\overline \mg}^{j}=\mg^{j} 
{-} \underbrace{  \frac{1}{\ns} \mG \ones_{_{[\ns]}}}_{\text{aver}(\bg^{j})}
{-} \frac{1}{\ns} \ones_{_{[\ns\!{\times}\!\ns]}} \mg^{j} %\\[-3ex]
+ \bigl(\underbrace{\frac{1}{\ns^{2}} \ones_{_{[\ns]}}^{\TT} \mG \ones_{_{[\ns]}} }_{\text{aver}(\mG)}\bigr) \ones_{_{[\ns]}}
\end{equation}
being $\mg^{j}$ the $j$-th column of $\mG$, $\ones_{_{[\ns]}} = [1, 1, \dots, 1]^{\TT}\in \RR^{\ns}$. 
 Expression \eqref{eq:centcol} is useful in the following, as a vector transformation. 
\NewText{It is worth mentioning that, as in the case of PCA, centering the kernel matrix is a best practice, but not mandatory. In the following, the notation $\overline{\mG}$ and $\overline{\mg}$ is adopted assuming that both the matrix and the vector are centered, see \eqref{eq:centerdmG2} and \eqref{eq:centcol}. However, the same expressions are also valid if they are not centered and they correspond directly to the raw kernel, see \eqref{eq:matmG2}.  }

\subsection{Diagonalization and dimensionality reduction}\label{sec:diagmG}
Recall that PCA is to be applied to the transformed sample $\mX$. However, $\mC$ is not available: only $\mG$ and its centered version $\overline{\mG}$ are computable. Thus, the idea is to use POD (diagonalize $\overline{\mG}$, as shown in section \ref{sec:permuting}) and invoque the idea of section \ref{sec:PCA3}.
The diagonalization of $\overline{\mG}$ yields
\begin{equation}\label{eq:diagmG}
\overline{\mG} = \mV \mtLLam \mV^{\TT} .
\end{equation}
The eigenvalues of $\mtLLam$, $\mlam_{1}\ge\mlam_{2}\ge \cdots \ge \mlam_{\ns}\ge0$, are expected to decrease such that the first $k$ (with $k\ll \min(d,\ns)$) collect a significant amount of the full variance, with a criterion associated with some small tolerance $\varepsilon$, see \eqref{eq:reducedk}.
Accordingly, the reduced version of $\mV$ is $\mV^{\star}$ of size $\ns \times k$ (taking the first $k$ columns of $\mV$). The dimensionality reduction of the samples is performed as shown is section \ref{sec:PCA3}. That is, the $k \times \ns$ matrix of samples in the reduced space $\ZZ^{\star}$  is computed as
\begin{equation}\label{eq:dimRedNonlinearM}
\ZZ^{\star}=\mV^{\star \, \TT}  \overline{\mG}
\end{equation}
This is equivalent to compute each vector $\zz^{j}$ (the $j$-th sample $\xx^{j}$ mapped into the reduced space) as 
\begin{equation}\label{eq:dimRedNonlinearv1}
\zz^{j} = \mV^{\star \, \TT} {\overline \mg}^{j}
\end{equation}
which is precisely the forward mapping into the reduced space $\RR^{k}$ of the sample $\xx^{j}\in\RR^{d}$.

The question arises on how to map a new element $\xx^{\new} \in \RR^{d}$ which does not belong to the initial training set. This mapping is described in both \eqref{eq:zstar} and \eqref{eq:zstar2} for the linear case (PCA). In the nonlinear case the forward mapping follows the idea of \eqref{eq:dimRedNonlinearv1} with a vector $\mg^{\new}\in\RR^{\ns}$ defined from the kernel and the elements of the sample (the training set). Vector $\mg^{\new}$ is such that its $i$-th component reads
\begin{equation}\label{eq:dimRedNonlinearv2}
[ \mg^{\new}]_{i}=\bk(\xx^{i},\xx^{\new}) .
\end{equation}
Vector $\mg^{\new}$ is centered
replacing $\mg^{j}$ by $\mg^{\new}$ in  expression \eqref{eq:centcol} to obtain ${\overline \mg}^{\new}$.
%\begin{multline}\label{eq:centcol2}
%{\overline \mg}^{\new}=\mg^{\new} 
%- ( \frac{1}{\ns}  \ones_{_{[\ns]}}^{\TT} \bg^{\new} ) \ones_{_{[\ns]}}\\
%- \frac{1}{\ns} \ones_{_{[\ns\!{\times}\!\ns]}} \mg^{\new} 
%+ (\frac{1}{\ns^{2}} \ones_{_{[\ns]}}^{\TT} \mG \ones_{_{[\ns]}} ) \ones_{_{[\ns]}}
%\end{multline}
The mapping of  $\xx^{\new}$ into the reduced space is finally obtained as
\begin{equation}\label{eq:dimRedNonlinearv3}
\zz^{\new} = \mV^{\star \, \TT} {\overline \mg}^{\new}
\end{equation}

%An illustration of the nonlinear multidimensionality reduction proposed by the kPCA is shown in Figure \ref{fig:scheme}.
Figure \ref{fig:scheme} shows how samples in the input space, $\xx^{j} \in \RR^{d}$ are mapped forward into $\zz^{j} \in \RR^{k}$, in the reduced space. This is ideally performed passing through the feature space, using PCA to reduce the dimensionality of the samples $\Phi(\xx^{j}) \in \RR^{D}$,  and then projecting them into the reduced space $\RR^{k}$ (grey dotted arrows in Figure \ref{fig:scheme}). In practice, the feature space is never used. With kPCA, the kernel trick introduces an alternative strategy (with no explicit definition of $\Phi$) that directly defines a mapping forward from $\xx^{j} \in \RR^{d}$ to $\zz^{j} \in \RR^{k}$ (represented by the blue dashed arrow in Figure \ref{fig:scheme}). 
\begin{figure}
	\centering
		\includegraphics[width=0.9\textwidth]{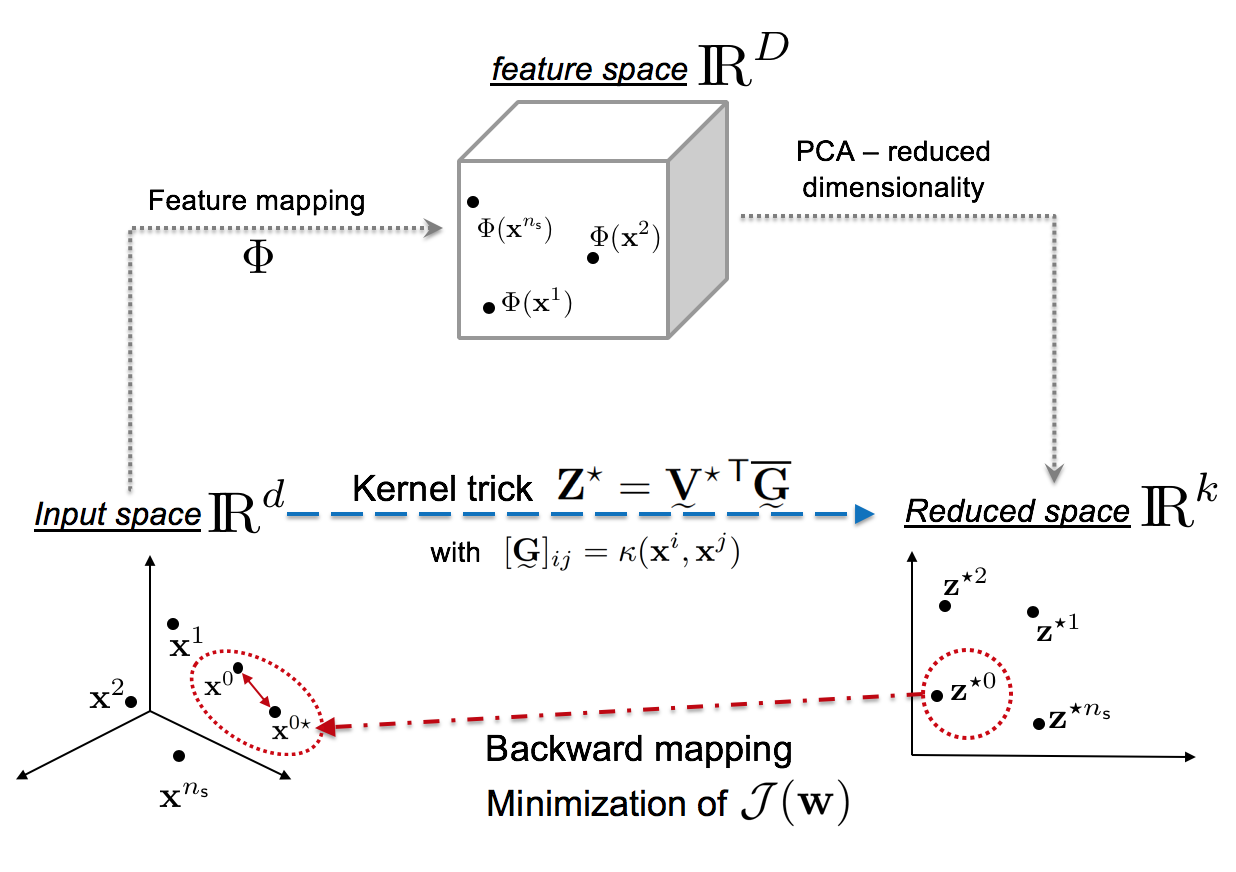}		
	\caption{Illustration the dimensionality reduction using kPCA.} 
	\label{fig:scheme}
\end{figure}

\section{kPCA backward mapping: from $\zz^{\star}$ to $\xx^{\star}$}\label{sec:backwardkPCA}
As noted in section \ref{sec:PCA3}, having only $\VV^{\star}$ does not allow properly mapping backwards an element $\zz^{\star}\in \RR^{k}$ into its pre-image $\xx^{\star}\in \RR^{d}$, (red dashed arrow in Figure \ref{fig:scheme}).

The typical strategy to perform the backward mapping consists in seeking the element $\xx^{\star}$ as a linear combination (or weighted average) of the elements of the training set, that is 
\begin{equation}\label{eq:weights1}
\xx^{\star}=\sum_{j=1}^{\ns} w_{j} \xx^{j}
\end{equation}
for some weights $w_{j}$, $j=1,\dots,\ns$. 
Note that the weights have to be computed as a function of the available data: that is, $\zz^{\star}$ and the samples in the training set. 
Often, the weights are explicitly computed in terms of the distances from $\zz^{\star}$ to all the samples mapped forward into the reduced space. That is, the distances $d_{i}=\Vert \zz^{\star} - \zz^{i} \Vert$ are computed in the reduced space and the weights $w_{i}$ are explicitly given as an inverse function of $d_{i}$ (the larger is the distance, the lower is the weight). Of course, this introduces some arbitrariness in the definition of the weights. There are many functions that decrease with the distance $d$: $w=1/d$; $w=1/d^{2}$; $w=\exp(-d)$... they all provide sensible results but with non-negligible discrepancies.

Here, we propose to select weights with an objective criterion, independent of any arbitrary choice. The idea is to adopt the form of \eqref{eq:weights1} and define the vector of unknown weights $\ww=[w_{1} \dots w_{\ns}]^{\TT}$. Then, vector ${\mg}^{\star}$ computed following \eqref{eq:dimRedNonlinearv2} becomes a function of the unknown weights as
%that can be explicitly computed for each value of $\ww$ using $\bk(\cdot,\cdot)$ 
\begin{equation}\label{eq:dimRedNonlinearv4}
[ \mg^{\star}]_{i}=\bk(\xx^{i},\xx^{\star})=\bk(\xx^{i},\sum_{\j=1}^{\ns} w_{j} \xx^{j}) .
\end{equation}
Vector $\mg^{\star}(\ww)$ is  centered using \eqref{eq:centcol} to become ${\overline \mg}^{\star}(\ww)$.
Then for a given value of $\zz^{\star}$, the following discrepancy functional is introduced:
\begin{equation}\label{eq:discrepancy}
{\mathcal J}(\ww)=\Vert \zz^{\star} - \VV^{\star \, \TT} {\overline \mg}^{\star}(\ww) \Vert^{2}
\end{equation}
Finally, the weights $\ww$ are selected such that they minimize the discrepancy functional,
\begin{equation}\label{eq:mini}
\ww = \arg \min_{\ww\in\RR^{\ns}} {\mathcal J}(\ww) .
\end{equation}
The minimization is performed with any standard optimization algorithm \cite{diez2003note,Ciarlet89}. In the example shown in the next section, the minimization method used is \texttt{fmincon}, interior-point algorithm implemented in \texttt{Matlab}, simply imposing the weights  to be nonnegative, that is $w_{j}\ge 0$.

In the case of the Gaussian kernel given in \eqref{eq:Gausskernel}, recalling that $\VV$ is a unit matrix (premultiplying by $\VV$ does not alter the norm)  and that the logarithm is monotonic, the previous discrepancy functional is replaced by an alternative form that measures discrepancy of the values of the logarithms
\begin{equation}\label{eq:discrepancy2}
{\mathcal J}(\ww)=\Vert \log( \VV^{\star}  \zz^{\star}) - \log({\overline \mg}^{\star}(\ww)) \Vert ^{2}.
\end{equation}
Note that functionals ${\mathcal J}$ in \eqref{eq:discrepancy} and \eqref{eq:discrepancy2} are different but they do lead to solutions minimizing the discrepancy of model and data.
\NewText{The strategy of using functional \eqref{eq:discrepancy2} precludes the possibility of centering the kernel as indicated in section \ref{sec:Centeringkernel}. That is, $\mG$ and $\mg$ must not be centered: with the centered versions $\overline{\mG}$ and $\overline{\mg}$, the arguments of the logarithms in \eqref{eq:discrepancy2}  are generally taking negative values.}

Using the expression of the Gaussian kernel,  \eqref{eq:dimRedNonlinearv4} becomes
\begin{equation}\label{eq:dimRedNonlinearv5}
[ \mg^{\star}]_{i}= \exp(-\beta\Vert \xx^{i}-\xx^{\star}\Vert^{2}) %\\[-3ex]
= \exp(-\beta\Vert \xx^{i}-(\sum_{j=1}^{\ns} w_{j} \xx^{j})\Vert^{2}) 
\end{equation}
and therefore the discrepancy functional ${\mathcal J}(\ww)$ adopts an explicit form in terms of the unknown weights $\ww$, namely
\begin{multline*}
{\mathcal J}(\ww)
=\sum_{i=1}^{\ns}  \bigr(\log([\VV^{\star}  \zz^{\star}]_{i})
                                + \beta\bigl\Vert \xx^{i}-(\sum_{j=1}^{\ns} w_{j} \xx^{j})\bigr\Vert^{2}\bigl)^{2} \\
=\sum_{i=1}^{\ns}  \bigr(\log([\VV^{\star}  \zz^{\star}]_{i})
                                + \beta[\Vert\xx^{i}\Vert^{2}-2\ww^{\TT} \XX^{\TT} \xx^{i} 
                                             + \ww^{\TT} \XX^{\TT}  \XX \ww]\bigr)^{2}  .
\end{multline*}

In order to ease the computational effort of the minimization algorithm it is convenient to reduce the number of unknown parameters by using some criterion based on the distance $d_{j}$. For instance, take as actual unknowns only the weights that correspond to samples with a distance in the reduced space lower than some threshold. This may drastically reduce the dimension of the minization problem from $\ns$ to the number of samples to be considered close enough to $\zz^{\star}$ (and hence to $\xx^{\star}$). 

\section{Numerical test}
The backward mapping proposed above (pre-image reconstruction) is tested in an example illustrated in Figure \ref{hands}. The sample  (training set) contains $\ns=69$ frames of a video consisting in an open hand that closes and opens once (nine of them shown in figure \ref{hands}).The number of grey-scale pixel values in each frame is $d= 1080 \times 1920 = 2073600$. An extra sample $\xx^{0}$ is kept apart to check the backward and forward mapping. Picture $\xx^{0}$ is a frame of the video outside the training set of $\ns$ samples and located in the center of the running time.
Intuitively,  this set of pictures can be described by a single parameter (the opening of the hand) and therefore the nonlinear manifold where the sample belongs is expected to be of dimension one.

\begin{figure}
	\centering
		\includegraphics[width=0.9\textwidth]{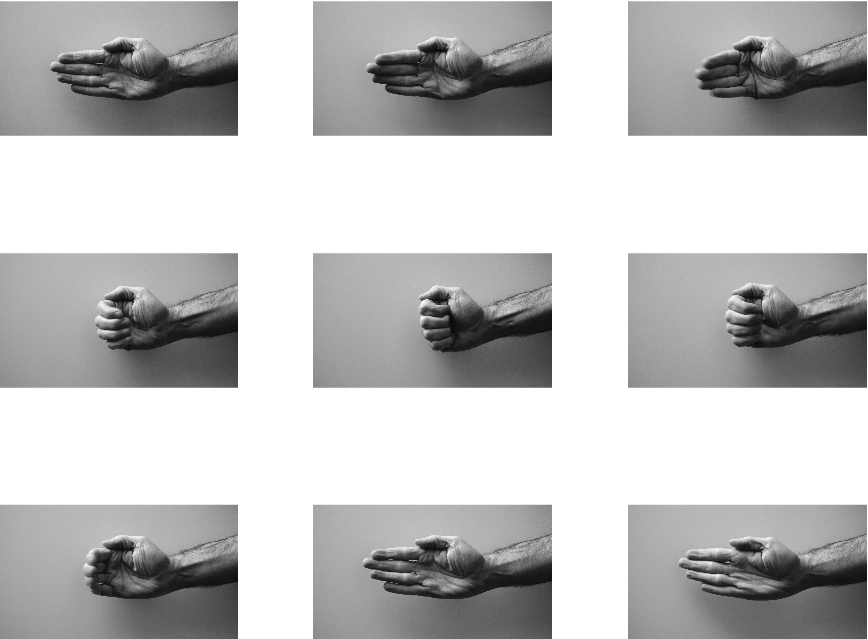}		
	\caption{Training set: 9 pictures are displayed (out of 69) of a hand in a sequence of closing and opening.} 
	\label{hands}
\end{figure}

The kPCA methodology described in section \ref{sec:kPCA} associated with an exponential kernel, 
%(taking $\beta=***$ in \eqref{eq:Gausskernel})
see \eqref{eq:Gausskernel}, provides a reduced model of three variables $k=3$ that collects approximately 50\% of the variance (corresponds to $\varepsilon=0.5$). This low level of accuracy is preferred in the example ($\varepsilon=0.5$ is too large) because $k=3$ is allowing a graphical illustration. Actually, figure \ref{fig:graph3D} shows (in blue) the representation of the  reduced samples $\zz^{j \star}$, $j=1,\dots,69$. 
Note that the configuration of the samples in the 3D space suggest that the actual intrinsic dimension of the manifold is one.
The image $\xx^{0}$ is mapped to the 3D space using the strategy devised in Section \ref{sec:diagmG} (see red dot in Figure \ref{fig:graph3D}).

\begin{figure}
	\centering
		\includegraphics[width=0.9\textwidth]{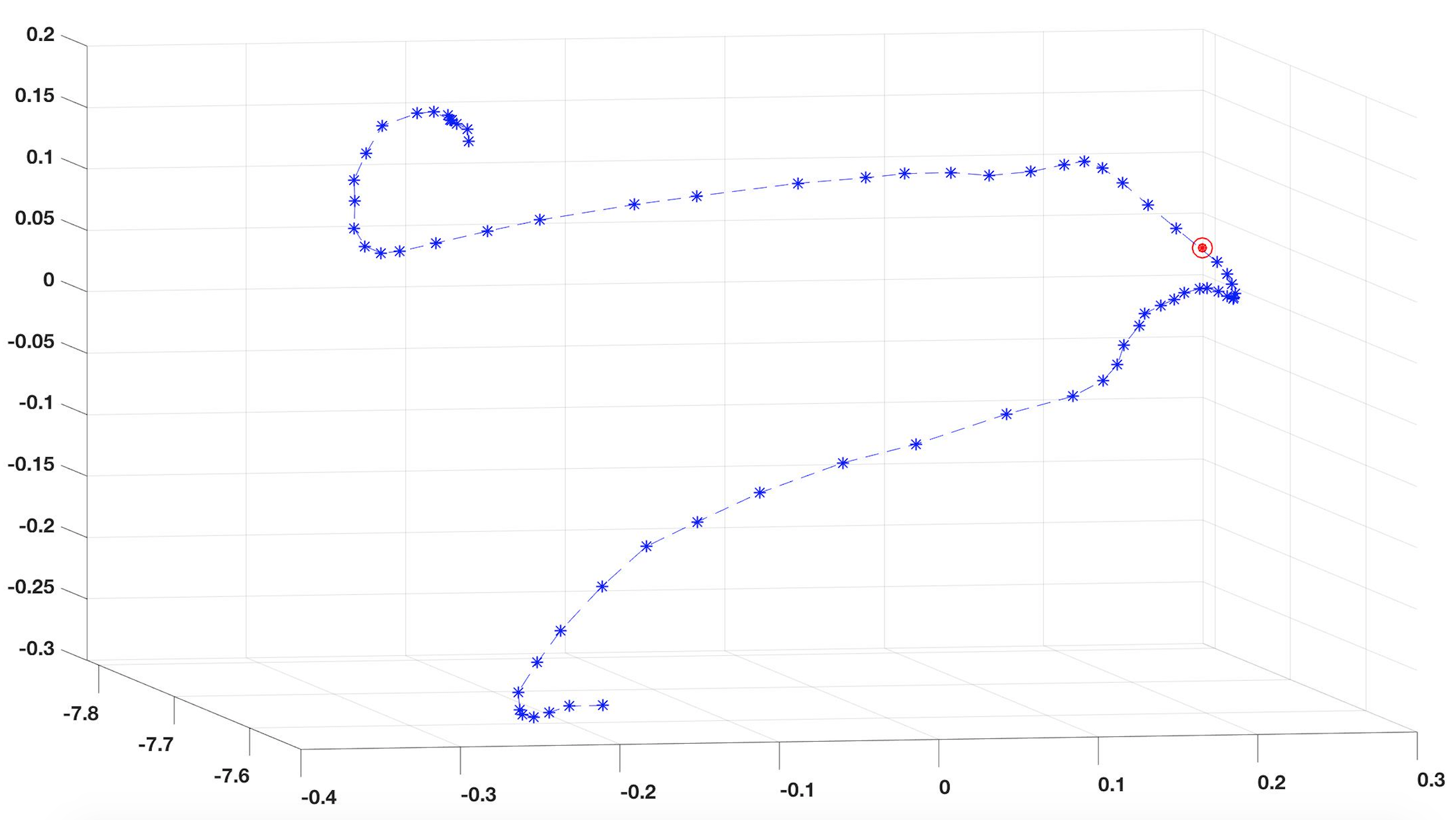}		
	\caption{Set of points representing the samples $\zz^{j \star}$, $j=1,\dots,\ns$ for $k=3$. The dashed line unites consecutive frames. The red dot is the forward mapping of the extra image $\xx^{0}$.} 
	\label{fig:graph3D}
\end{figure}

\begin{figure}
	\centering
		\includegraphics[width=0.9\textwidth]{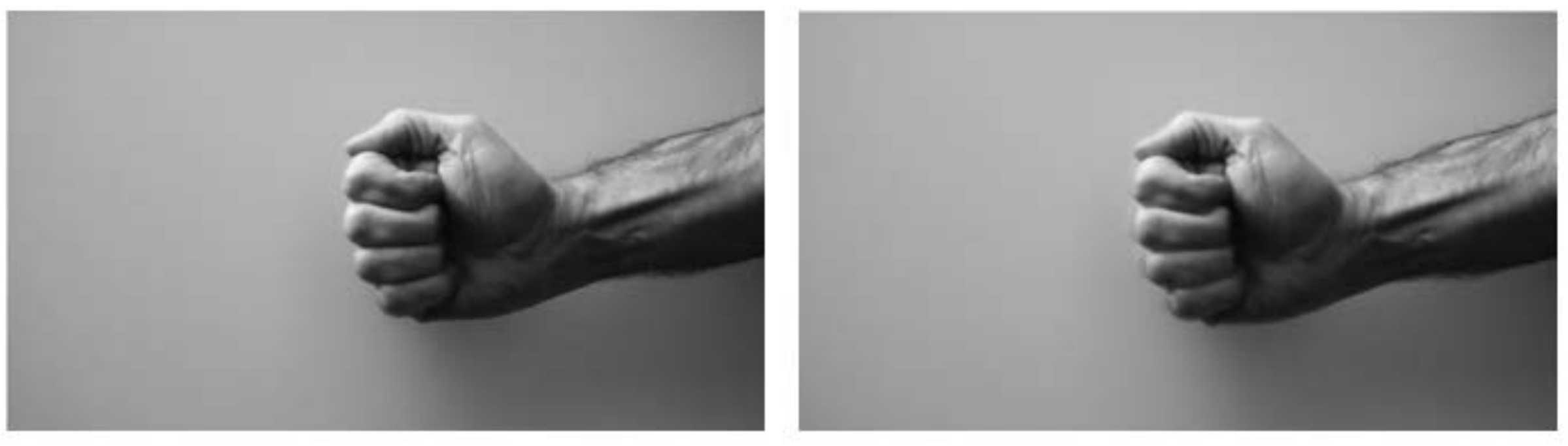}		
	\caption{Original picture $\xx^{0}$ out of training set  (left) and its pre-image approximation  $\xx^{0 \star}$(right, after consecutive forward and backward mapping). } 
	\label{fig:ex36}
\end{figure}

The new sample in the reduced space is mapped backward to the input space $\RR^{d}$ following the ideas in Section \ref{sec:backwardkPCA}. This produces the pre-image $\xx^{0 \star}$ which is an approximation to $\xx^{0}$, see Figure \ref{fig:ex36}.
Note that the dimensionality reduction (from 2073600 degrees of freedom to 3) is preserving most of the features of the image.

\begin{appendices}

\section{Three vectors product rule}\label{app2}
All along the paper, the following product rule is extensively used. 
For three arbitrary vectors $\va$ and $\vb$ and $\vc$ in  $\RR^{d}$, it holds that
\begin{equation}\label{eq:threeVectProd1}
(\va\vb^{\TT})\vc=\va(\vb^{\TT}\vc)
\end{equation}
or, using the tensorial and scalar product notation
\begin{equation}\label{eq:threeVectProd2}
(\va \otimes \vb)\vc=\va (\vb \cdot \vc)
\end{equation}
This is easily shown considering the index notation
\begin{equation}\label{}
[(\va \vb^{\TT})\vc]_{i}=[\va \vb^{\TT}]_{ik}\vc_{k}=\va_{i} \vb_{k}\vc_{k} =[\va(\vb^{\TT}\vc)]_{i} ,
\end{equation}
where the Einstein summation convention is used.

\section{The first $d$ columns of $\VV$ are equal to $\BB$ }\label{app:proof1}
\begin{prop}
The $\ns \times d$ matrix $\BB$ introduced in \eqref{eq:matB} coincides with the first $d$ columns of matrix $\VV$.
\end{prop}
\begin{proof}
The $\ns \times d$ matrix $\BB$ is
\begin{equation*}
[\BB]_{\ell j} =\frac{1}{\lambda_{j}} \xx^{\ell \, \TT} \uu^{j}.
\end{equation*}
for $ \ell=1,\dots,\ns$ and $j=1,\dots,d$.
As shown in \eqref{eq:ueqsBx}
\begin{equation*}
\uu^{j}= \sum_{\ell=1}^{\ns} [\BB]_{\ell j} \xx^{\ell}
\end{equation*}
Thus, recalling that 
$\CC =\XX \XX^{\TT} =  \sum_{\ell=1}^{\ns} \xx^{\ell} \xx^{\ell \, \TT} $
the eigenvalue problem \eqref{eq:eig1} results in 
\begin{equation}\label{eq:proof1}
%\left( 
(\sum_{\ell=1}^{\ns} \xx^{\ell} \xx^{\ell \, \TT} )
%\right)
\sum_{\hl=1}^{\ns} [\BB]_{\hl j} \xx^{\hl} 
=
\lambda_{j}  \sum_{\hl=1}^{\ns} [\BB]_{\hl j} \xx^{\hl}
\end{equation}
The left-hand-side of \eqref{eq:proof1} becomes
\begin{multline*}
 \sum_{\ell=1}^{\ns}   \sum_{\hl=1}^{\ns}  [\BB]_{\hl j}  (\xx^{\ell} \xx^{\ell \, \TT})  \xx^{\hl}
 =  \sum_{\ell=1}^{\ns}   \sum_{\hl=1}^{\ns}  [\BB]_{\hl j}  (\xx^{\ell , \TT} \xx^{\hl})  \xx^{\ell} \\
 =  \sum_{\ell=1}^{\ns}   \sum_{\hl=1}^{\ns}  [\BB]_{\hl j} [\GG]_{\ell \hl} \xx^{\ell}
 =  \sum_{\ell=1}^{\ns}  \xx^{\ell}  \sum_{\hl=1}^{\ns}  [\BB]_{\hl j} [\GG]_{\ell \hl} 
\end{multline*}
And left multiplying \eqref{eq:proof1} by $\xx^{\gamma \, \TT}$ yields
\begin{align*}
 \sum_{\ell=1}^{\ns} 
\underbrace{\xx^{\gamma \, \TT}  \xx^{\ell}  }_{[\GG]_{\gamma \ell}}
\sum_{\hl=1}^{\ns}  [\BB]_{\hl j} [\GG]_{\ell \hl} 
&= \lambda_{j}  \sum_{\hl=1}^{\ns} [\BB]_{\hl j} 
\underbrace{\xx^{\gamma \, \TT}  \xx^{\hl}  }_{[\GG]_{\gamma \hl}} \\
\Rightarrow  \sum_{\hl=1}^{\ns}  [\BB]_{\hl j}  \sum_{\ell=1}^{\ns} [\GG]_{\gamma \ell} [\GG]_{\ell \hl} 
&= \lambda_{j}  \sum_{\hl=1}^{\ns} [\BB]_{\hl j} [\GG]_{\gamma \hl} \\
\Rightarrow  \sum_{\hl=1}^{\ns}  [\BB]_{\hl j} [\GG^{2}]_{\gamma \hl}
&= \lambda_{j}  \sum_{\hl=1}^{\ns} [\BB]_{\hl j} [\GG]_{\gamma \hl} 
\end{align*}
The previous expression is written in matrix form as 
\begin{equation*}
\GG^{2} \BB = \GG \LLam \BB \Rightarrow \GG \BB = \LLam \BB
\end{equation*}
This proves that the $d$ columns of $\BB$ are indeed eigenvectors of $\GG$ associated with eigenvalues $\lambda_{i}$, for $i=1,\dots,d$ and therefore they do coincide with the first $d$ columns of $\VV$ (possibly up to a normalization constant).
\end{proof}

\end{appendices}

\section*{Acknowledgments}
This work is partially funded by Generalitat de Catalunya (grant number 1278 SGR 2017-2019) and 
Ministerio de Econom\'ia y Empresa and Ministerio de Ciencia, Innovaci\'on y Universidades (grant number DPI2017-85139-C2-2-R.
\medskip
\medskip
\bibliographystyle{plain}
\bibliography{Ref}

\end{document}